\documentclass{amsart}

\usepackage{amssymb,amsmath,amsthm}

\newcommand\colo{\colon\,}
\newcommand{\nqed}{\renewcommand\qed{}}

\newenvironment{mylist}
  {\begin{enumerate}\setlength{\itemsep}{0mm}}{\end{enumerate}}

\theoremstyle{plain}
\newtheorem{theorem}{Theorem}
\newtheorem{proposition}[theorem]{Proposition}
\newtheorem{lemma}[theorem]{Lemma}
\newtheorem{corollary}[theorem]{Corollary}
\theoremstyle{definition}
\newtheorem{definition}[theorem]{Definition}

\newtheorem{remark}[theorem]{Remark}
    
\newcounter{listitem}

\newenvironment{thalphalist}{\begin{list}{\upshape(\alph{listitem})}
        {\setlength{\itemsep}{0mm}\usecounter{listitem}} }
    {\end{list}}
    
\newcounter{eqlist}

\newcommand{\Iff}{\Leftrightarrow}

\newcounter{to}
\newenvironment{toenum}{\begin{enumerate}
        \setlength{\itemindent}{1.94em}
        \setcounter{enumi}{\value{to}}      

        \setlength{\itemsep}{0pt}}
{\setcounter{to}{\value{enumi}}\end{enumerate}}
\newcommand{\rto}[1]{($\to_{\ref{#1}}$)}

\newcommand\thitem[1]{\item[\ {\upshape (#1)}]}

\author{J\=anis C\={\i}rulis}
\address{University of Latvia, Riga, Latvis}

\title{IPLICATIONS IN SECTIONALLY PSEUDOCOMPLEMENTED POSETS}
\keywords{BCK-algebras, commutative groupoid, condition S, equational class, implication, relative pseudocomplement, residuation, sectional pseudocomplement}

\subjclass{Primary 03G25; secondary 06A12, 06D15, 06F35}

\begin{document}

\maketitle

\begin{abstract}
A sectionally pseudocomplemented poset $P$ is one which has the top element and in which every principal order filter is a pseudocomplemented poset. The sectional pseudocomplements give rise to an implication-like operation on $P$ which coincides with the relative pseudocomplementation if $P$ is relatively psudocomplemented. We characterise this operation and study some elementary properties of upper semilattices, lower semilattices and lattices equipped with this kind of implication. We deal also with a few weaker versions of implication.
Sectionally pseudocomplemented lattices have already been studied in the literature.
\end{abstract}

\maketitle

\section{Introduction}

This study is roused by the paper \cite{Ch} the subject of which is lattices with the largest element and pseudocomplemented upper sections (principal filters). Such a lattice $(L, \wedge, \vee, 1)$ admits a partial binary operation $*$ defined as follows:
\begin{equation}	\label{ast}
x * y = z \mbox{ if and only if } y \le x \mbox{ and } z \mbox{ is the pseudocomplement of } x \mbox{ in } [y) .
\end{equation}
Put in another way, this means that $x * y$ is defined if and only if $y \le x$, and
\begin{equation}	\label{ast1}
\mbox{if } x \in [y), \mbox{ then, for all } u \in [y), \
u \le x * y \mbox{ if and only if } u \wedge x = y .
\end{equation} 
In particular, if $y \le x$ and $z$ is the pseudocomplement of $x$ relatively to $y$, i.e., if
\[
\mbox{for all } u, \ u \le z  \mbox{ if and only if } u \wedge x = y .
\]
then $z = x * y$.
The algebra $(L, \wedge, \vee, *, 1)$ could be called a \emph{sectionally pseudocomplemented lattice}. 
The total binary operation $\to$ defined on $L$ by the condition
\begin{equation}	\label{to/*}
 x \to y := (x \vee y) * y
 \end{equation}
is, evidently, an extension of $*$.
Sectionally pseudocomplemented lattices and their extensions are explored further in \cite{ChR1,ChR2}. As noted in \cite[Remark 2.2]{ChR1}, the extension $(L,\vee,\wedge,\to,1)$ of a distributive sectionally pseudocomplemented lattice is a Brouwerian lattice (Heyting algebra). 

Another type of extension of the operation $*$:
\begin{equation}	\label{to1/*}
x \to y : = x * (x \wedge y)
\end{equation}
was investigated in \cite{ChH,HK}. (It should also be noted that meet semilattices with pseudocomplemented lower sections (principal ideals) have been studied already in 
\cite{K1,K2,MM}.)

A natural way to extend the notion of a pseudocomplementation to arbitrary posets has been discovered by several authors --- see \cite{K,V,H}. Correspondingly, sectional pseudocomplements can also be considered in posets that are not meet semilattices. By understanding pseudocomplements in (\ref{ast}) in this wider sense, we obtain, instead of (\ref{ast1}), the following characteristic condition for sectional pseudocomplementation $*$ in a poset 
 (we write $u \perp_y x$ to mean that there is no lower bound of $u$ and $x$ in $[y)$ distinct from $y$): 
\begin{equation}	\label{ast2}
\mbox{if } x \in [y), \mbox{ then, for all } u \in [y), \
u \le x * y \mbox{ if and only if } u \perp_y x .
\end{equation} 
This condition reduces to (\ref{ast1}) if the poset is a meet semilattice. 

The concept of pseudocomplementation may be further weakened in various ways. In this paper, we  consider posets (in particular, semilattices and lattices) with  weakened sectional pseudocomplementation (\ref{ast2}), as well as its extension allied to (\ref{to/*}).

\section{Extensions of sectionally pseudocomplemented posets}

Now suppose that  $(A,\vee,1)$ is a semilattice with unit, and $*$ is the operation defined by (\ref{ast2}).
The extension  of $*$ given by (\ref{to/*}) can, actually, be defined without an explicit reference to $*$. We first note that, in $A$, the condition $u \perp_y x$ actually means that $u \wedge x$ exists and equals to $y$. Indeed, if $y \le u,x$, then
\begin{eqnarray*}
u\, \perp_y\, x 
& \Iff & 
\forall v \in [y)\, (\mbox{if } v \le u, x, \mbox{ then } v = y)	\\
& \Iff & 
\forall w (\mbox{if } w \vee y \le u, x , \mbox{ then } w \vee y = y)	\\
& \Iff & 
\forall w (\mbox{if } w \le u, x, \mbox{ then } w \le y)	\\
& \Iff & 
\forall w (w \le u, x \mbox{ iff } w \le y). 
\end{eqnarray*}
Furthermore,
\[
z \le (x \vee y) * y  \Iff  
 z \vee y \le (x \vee y) * y \Iff  z \vee y\, \perp_y\, x \vee y
\] 
for all $z \in A$. These observations lead us to the following lemma.

\begin{lemma}	\label{L1}
A binary operation $\to$ on $A$ satisfies {\upshape(\ref{to/*})} if and only if, for all $x$, $y$ and $z$ in $A$,
\begin{equation}	\label{to/vee}
z \le x \to y = (z \vee y) \wedge (x \vee y) \mbox{ exists and equals to } y.
\end{equation}
\end{lemma}

The next theorem shows that the operation $\to$ can be characterised even without  reference to join.

\begin{theorem} A binary operation $\to$ on $A$ satisfies {\upshape (\ref{to/*})} if and only if it  has the following properties:
\begin{toenum}
\item   \label{wexch}
if $x \le y \to z$, then $y \le x \to z$,   
\item 	\label{reg}
if $x \le x \to y$, then $x \le y$,
\item 	\label{comp}
if the meet of $x$ and $y$ exists, then $x \le y \to ( x \wedge y)$.	
\end{toenum}
\end{theorem}
\begin{proof}
It is easily seen that \rto{wexch}, \rto{reg} 
and \rto{comp} hold for the operation $\to$ characterised by (\ref{to/vee}). 

Conversely, if the operation $\to$ satisfies 
the conditions \rto{wexch}--\rto{comp} and $*$ is its restriction defined by
\begin{equation}	\label{*/to}
x * y = v \mbox{ if and only if } y \le x \mbox{ and } v = x \to y,
\end{equation}
then, obviously,  (\ref{to/*}) holds true. Let us see, why $x * y$ is the pseudocomplementation of $x$ in $[y)$.
Suppose that $u,x \in [y)$. If $u \le x * y$ and 
$y \le v \le u,x$, then $v \le u \le x \to y$ and, furthermore, $v \le x \le v \to y$ by \rto{wexch}, wherefrom $v \le y$, i.e. $v = y$ by \rto{reg}. 
If, conversely, $u \perp_y x$, then $y$ is the greatest lower bound of $x$ and $y$ in $[y)$ and (as noted at the beginning of the section) even in $L$, and then $z \le x \to y= x * y$ by \rto{comp}. So, (\ref{to/vee}) holds by Lemma \ref{L1}.
\end{proof}

Now let $(A,\to,1)$ be any algebra in which $A$ is a poset with 1 the greatest element and $\to$ is a binary operation  obeying the conditions \rto{wexch}--\rto{comp}. It is an implicative algebra in the sense of \cite{R}, for the relationship 
\begin{equation}	\label{le/to}
x \le y \mbox{ if and only if } x \to y = 1
\end{equation}
is an easy consequence of these conditions. Indeed, if  $1 \le x \to y$, then $x \le x \to y$, and the inequality $x \le y$ follows by \rto{reg}. Conversely, suppose that $x \le y$. By \rto{comp}, $1 \le y \to y$, and then $x  \le 1 \to y$ in virtue of \rto{wexch}.

We know from the proof of the theorem that $A$ is an extension of a sectionally pseudocomplemented poset. It follows from Corollary \ref{unique} below that this extension is completely determined by the underlying poset.
To remind that the characterised properties of $\to$ were based on (\ref{to/*}) rather than on (\ref{to1/*}), this kind of extension could even be termed a j-extension (`j' for `join').

These observations  motivate the following definition.

\begin{definition}
An algebra $(A,\to,1)$ satisfying \rto{wexch}--\rto{comp} is said to be a \emph{sectionally j-pseudocomplemented poset}. The operation $\to$ itself is called \emph{sectional j-pseudocomplementation}. 
\end{definition} 

Sectionally j-pseudocomplemented semilattices and lattices are defined similarly.

\section{Weak BCK*-algebras}

It turns out that many important properties of sectionally j-pseudocomplement\-ation actually do not depend of \rto{reg} and \rto{comp}. See Remark \ref{bckrem} for the motivation of the term `wBCK*-algebra' used in the subsequent definition.

\begin{definition}
A \emph{weak BCK*-algebra}, or just \emph{wBCK*-algebra}, is an implicative algebra $(A,\to,1)$, where  $\to$ satisfies \rto{wexch}.
A wBCK*-algebra is said to be \emph{weakly contractive} if it satisfies \rto{reg}.
\end{definition}

\begin{lemma}
In every wBCK*-algebra,
\begin{toenum}
\item $x \le (x \to y) \to y$,	\label{bck2}
\item if $x \le y$, then $y \to z \le x \to z$. \label{bck1-}   
\item   \label{ubound}
$y \le (x \to y) \to y$, 
\item   \label{expan}
$((x \to y) \to y) \to y = x \to y$,                  
\item $x \to x = 1$,	\label{refl}
\item   \label{h1}
$x \le y \to x$,    
\item $1 \to x = x$,	\label{Sreg}
\item $x \to 1= 1$.	\label{unit}
\end{toenum}
\end{lemma}
\begin{proof}
\rto{bck2} Trivially.
\par
\rto{bck1-} By \rto{bck2} and \rto{wexch}.
\par
\rto{ubound} Similarly.
\par
\rto{expan}
By \rto{bck2} and \rto{bck1-}, $((x \le y) \to y) \to y \le x \to y$. The converse inequality is a particular case of  \rto{bck2}. 
\par
\rto{refl} Follows from (\ref{le/to}).
\par
\rto{h1}
By \rto{wexch} and \rto{refl}, as $y \le 1 = x \to x$. \par
\rto{Sreg} 
By (\ref{le/to}), the inequality $1\to x \le x$ follows from \rto{bck2}. 
Its converse is a particular case of \rto{h1}.
\par
\rto{unit} Follows from \rto{h1}.
\end{proof}

Now it can be shown that the structure of every wBCK*-algebra is completely determined by the structure of its sections. In particular, a sectionally pseudocomplemented poset admits at most one wBCK*-algebra extension.

\begin{lemma} \label{unique}
Suppose that $(A,\to,1)$ is a wBCK*-algebra and that $*$ is the restriction  of $\to$ determined by {\upshape (\ref{*/to})}. Then 
\[
x \to y = \max\{z * y\colo x,y \le z\}.
\]	
\end{lemma}
\begin{proof}
Let $z := (x \to y) \to y$. Then $x,y \le z$ by \rto{bck2} and \rto{ubound}, and further $x \to y = z \to y = z * y$ by \rto{expan} and (\ref{*/to}). On the other hand, if $x,y \le z$, then $z * y = z \to y \le x \to y$ by \rto{bck1-}.
\end{proof}

\begin{remark} \label{bckrem}
It follows from \rto{bck2} and \rto{bck1-} that $y \le (y \to z) \to z \le x \to z$ whenever $x \le y \to z$. Therefore, this pair of conditions is equivalent to the   axiom \rto{wexch} of wBCK*-algebras. The latter term is motivated by this observation: if the algebra $(A,\to,1)$ satisfies  (\ref{le/to}), \rto{bck2} and the following strengthening of \rto{bck1-}
\[
x \to y \le (y \to z) \to (x \to z), \label{bck1}
\]
then it is the dual algebra (w.r.t.\ the ordering $\le$) of a BCK-algebra (see, e.g., \cite{IT}). We adopt the asterick notation BCK* for such duals from  \cite{Tor}. 
\end{remark}

A Hilbert algebra can be characterised as a positive implicative BCK*-algebra, i.e., a BCK*-algebra in which
\[
x \to (y \to z) \le (x \to y) \to (x \to z);
\]
see \cite{Co1,Tor,Ko}. The latter condition may be replaced by the identity
\[
x \to (x \to y) = x \to y
\]
(\cite[Theorem 8]{IT}; see also \cite[Theorem 1]{Tor}. Since this identity covers \rto{reg}, every Hilbert algebra is an example of a weakly contractive wBCK*-algebra. See Corollary \ref{wbckhilb} below for a stronger result.

A particular kind of Hilbert algebras are relatively pseudocomplemented posets \cite{Ru}:  algebras $(A,\to,1)$, where $A$ is a poset with 1 the maximum element and the operation $\to$ that satisfies the condition
\begin{equation}	\label{1+2}
\mbox{if } u \le x \to y \mbox{ and } v \le x,u, \mbox{ then } v \le y 
\end{equation}
as well as its converse---the following strengthening of \rto{comp}:
\begin{equation} \label{ccomp}
\mbox{if } v \le y \mbox{ whenever } v \le x,u, \mbox { then } u \le x \to y.
\end{equation}
In fact, (\ref{1+2}) is equivalent to its particular case
\begin{equation} \label{reg-}
\mbox{if } z \le x, z \le x \to y, \mbox{ then } z \le y.
\end{equation}
The subsequent lemma characterises the relation between sectionally j-pseudocom\-plemented and relatively pseudocomplemented posets more exactly.

\begin{theorem}
Let $A$ be a poset with the greatest element $1$. A binary operation $\to$ on $A$  is relative pseudocomplementation if and only if  it 
satisfies {\upshape\rto{wexch}, \rto{reg}} and \upshape{(\ref{ccomp})}.
\end{theorem}
\begin{proof}
It was established in the proof of Lemma \ref{L1} that  
(\ref{1+2}) 
is a consequence of \rto{wexch} and \rto{reg}. 
On the other hand, the  conditions \rto{wexch} and \rto{reg} are fulfilled in every relatively pseudocomplemented poset. Indeed, \rto{reg} is a particular case of (\ref{1+2}) with $u=v=x$. To prove \rto{wexch}, assume that $x \le y \to z$. By (\ref{1+2}), then $v \le x,y$ implies that $v \le z$. Therefore $y \le x \to z$ by (\ref{ccomp}).
\end{proof}

We shall say that a wBCK*-algebra $A$ is an upper (lower) \emph{wBCK*-semilattice} or a \emph{wBCK*-lattice}, if $A$  happens to be an upper (lower) semilattice or a lattice, respectively. 
A Hilbert algebra with infimum \cite{FRS}, i.e., a lower semilattice-ordered Hilbert algebra, is an example of a weakly contractive lower wBCK*-semilattice. BCK*-semilattices and lattices have been studied in \cite{Idz}. Relatively pseudocomplemented semilattices, known also as Brouwerian or implicative semilattices, form a subclass of sectionally j-pseudocomplemented lower BCK*-semilattices. Likewise, relatively pseudocomplemented, or implicative,  lattices (Heyting algebras) form a subclass of BCK*-lattices.

\begin{theorem}
A sectionally j-pseudocomplemented lower semilattice (lattice) is relatively pseudocomplemented if and only if it satisfies the condition 
\begin{toenum}
\item	\label{toisot}
if $x \le y$, then $z \to x \le z \to y$ .
\end{toenum}
\end{theorem}
\begin{proof}
Due to (\ref{ccomp}) and(\ref{reg-}), the condition (\ref{toisot}) is fulfilled in every relatively pseudocomplemented poset.
Now assume that a wBCK*-algebra $A$ is a lower semilattice satisfying (\ref{toisot}); 
by the previous theorem it suffices  to prove only that (\ref{ccomp}) holds. Let $v \le x,u$ implies, for every $v$, that $v \le y$. Then
$u \wedge x \le y$ and, by \rto{comp} and \rto{toisot}, \end{proof}

As \rto{toisot} holds in all BCK*-algebras (see \cite[Theorem 2]{IT}), we conclude that a lower BCK*-semilattice  is relatively pseudocomplemented (i.e., is an implicative semilattice) if and only if it is  sectionally j-pseudocomplemented.

\section{Weak BCK*-algebras with condition S}

Adapting the definition known for BCK-algebras (see, e.g., \cite{Is,IT}), we shall say that a wBCK*-algebra $A$ \emph{satisfies condition S} if, for all $x$ and $y$, the subset $\{z\colo x \le y \to z\}$ has the least element. 
We may denote it by $x \cdot y$; this way a binary operation $\cdot$ on $A$ may be introduced. The couple of operations $(\cdot, \to)$ illustrates the following definition.

\begin{definition}
By a \emph{(binary) adjunction} on a poset $A$ we mean a pair $(\cdot,\to)$ of binary operations on $A$ satisfying the condition 
\begin{equation}    \label{cdot/to}
x \le y \to z \mbox{ if and only if } xy \le z.
\end{equation}
\end{definition}

\begin{proposition}	\label{adj4}
An adjunction $(\cdot,\to)$ can equivalently be characterised by four conditions
\begin{thalphalist}
\thitem{a1}
$x \le y \to xy$,
\thitem{a2}
$(x \to y)x \le y$,
\thitem{a3}
if $x \le y$, then $z \to x \le z \to y$,
\thitem{a4}
if $x \le y$ then $xz \le yz$.
\end{thalphalist}
\end{proposition}

Note that (a3) coincides with \rto{toisot}.

 \begin{theorem} \label{dual} 
 Suppose that $(\cdot,\to)$ is an adjunction on a poset $A$ and $1 \in A$.
 Then the following statements are equivalent:
\begin{thalphalist}
\item $(A,\to,1)$ is a wBCK*,
\item $(A,\cdot,1)$ is a commutative groupoid with the neutral element 1 which is also the largest element in $A$.
\end{thalphalist}
If it is the case, then the wBCK*-algebra is weakly contractive if and only if the groupoid is idempotent.
\end{theorem}
\begin{proof}
(a) $\to$ (b).
Commutativity of $\cdot$ follows from (\ref{cdot/to}) in virtue of \rto{wexch}. By \rto{unit}, $1$ is the maximum element in $A$. By \rto{refl} and (\ref{cdot/to}), $1\cdot x \le x$. At last, $x \le 1\cdot x$ by (a1) and (\ref{le/to}). 

(b) $\to$ (a). 
The two assumptions on $1$ provide (\ref{le/to}): 
\[
x \to y = 1 \Iff 1 \le x \to y \Iff  1\cdot x \le y \Iff x \le y \, .
\]
We know from Remark \ref{bckrem} that \rto{wexch} is a consequence of \rto{bck2} and \rto{bck1-}. 
The property (a2) together with commutativity of $\cdot$ allows us to prove \rto{bck2}. To obtain \rto{bck1-}, assume that $x \le y$.  As $\cdot$ is commutative, it follows from (a4) that $ux \le uy$ for all $u$. Then, for every $z$, $uy \le z$ implies that $ux \le z$. Therefore, $u \le y \to z$ implies that $u \le x \to z$. Hence, $y \to z \le x \to z$.

For the last assertion note that
\[
x \le x \to y \Iff x \cdot x \le y \Iff x \le y,
\] 
if $\cdot$ is idempotent, and that
\[
x \cdot x \le y \Iff x \le x \to y \Iff x \le y ,
\]
if $\to$ is weakly contractive (the condition ``if $x \le y$, then $x \le x \to y$'' inverse to \rto{reg} follows from \rto{h1}). Therefore, idempotency of $\cdot$ turns out to be equivalent to condition that $\to$ has to be weakly contarctive. 
\end{proof}

In virtue of (a4), if the operation $\cdot$ in an adjunction is commutative, then it gives rise to a partially ordered groupoid (po-groupoid). A commutative  po-groupoid is said to be \emph{integral}, if it has the neutral element which is also the maximum element, and \emph{residuated} if the multiplication $\cdot$ has the adjoint operation $\to$. We shall use the acronym \emph{pocrig} for a partially ordered commutative residuated and integral groupoid. Therefore, a pocrig can be viewed as an algebra of type $(A,\cdot,\to,1)$. A pocrig with  associative multiplication is known as a \emph{pocrim}; see \cite{BF,BR}.

\begin{corollary}
An algebra $(A,\to,1)$ is a wBCK*-algebra with condition S if and only if it is a reduct of a pocrig.
\end{corollary}

A similar correspondence between BCK*-algebras with condition S and pocrims has already be noticed in the literature; see, e.g. \cite{BR}. It should be noted that some authors include multiplication in the signature of BCK*-algebras with condition S; then the class of BCK*-algebras with condition S coincides with the class of pocrims.

The last assertion of Theorem \ref{dual} suggests that idempotent pocrigs and regular wBCK*-algebras with condition S should be related to each other in the same way. In fact, we can say more about this situation.

The next lemma (which slightly improves Lemma 4.1 in \cite{OK}) implies that such a wBCK*-algebra is even relatively pseudocomplemented (hence, a Hilbert algebra --- see the preceding section). 

\begin{lemma}
An idempotent pocrig $(A,\cdot,\to,1)$ is an implicative semilattice, i.e., $(A,\cdot,1)$ is a lower semilattice with unit, and $\to$ is relative pseudocomplementation on $A$.
\end{lemma}
\begin{proof}
As multiplication in a pocrig is, by definition, isotone, it follows from $x,y \le 1$ that $x \cdot y$ is a lower bound of $x$ and $y$. Assume that $z \le x,y$; then $x \cdot z \le x \cdot y$ and $z = z \cdot z \le x \cdot z$. Therefore, $z \le x\cdot y$, and $x \cdot y$ is actually the greatest lower bound of $x$ and $y$. Then the adjoint $\to$ of $\cdot$ becomes relative pseudocomplementation on $A$. 
\end{proof}

\begin{corollary} 
An algebra $(A,\to,1)$ is a weakly contractive wBCK*-algebra with condition S if and only if it is a reduct of an implicative semilattice.
\end{corollary}

Hilbert algebras are just subreducts of implicative semilattices \cite[Theorem 12]{D1} (see also  Theorem 8 of \cite{Ho}). This gives us the following characteristic of those wBCK*-algebras that are Hilbert algebras.

\begin{corollary}	\label{wbckhilb}
A wBCK*-algebra is a Hilbert algebra if and only if it is a subalgebra of a weakly contractive wBCK*-algebra with condition S.
\end{corollary}

By a  \emph{multiplicative semilattice} we, following \cite{B}, shall mean an upper semilattice with multiplication which is both left and right distributive.

\begin{theorem}
 Suppose that $(\cdot,\to)$ is an adjunction on a poset $A$, $1 \in A$, and $\vee$ be a binary operation on $A$.
 Then the following statements are equivalent:
\begin{thalphalist}
\item $(A,\vee, \to,1)$ is a wBCK*-semilattice, 
\item $(A,\cdot,1)$ is a semilattice ordered commutative integral groupoid,
\item $(A,\vee,\cdot, 1)$ is an integral and commutative multiplicative semilattice.
\end{thalphalist}
\end{theorem}
\begin{proof}
In virtue of Theorem \ref{dual}, it remains to show that a semilattice ordered pocrig is distributive. For all $u \in A$,
\begin{eqnarray*}
(x \vee y)z \le u & \Iff & x \vee y \le z \to u \\
 & \Iff & x \le z \to u \mbox{ and } y \le z \to u	\\
 & \Iff & xz \le u \mbox{ and } yx \le u	\\
 & \Iff & xz \vee yz \le u. \quad \qed
   \end{eqnarray*}
   \nqed
\end{proof}

\begin{corollary}
An algebra $(A, \vee, \to, 1)$ is an upper wBCK*-semilattice with condition S  if and only if it is a reduct of a semilattice ordered pocrig or, equivalently, of a residuated integral multiplicative semilattice.
\end{corollary}

\section{Some equational classes of expanded wBCK*-algebras}

It is well-known \cite{W} that the class of BCK*-algebras is not a variety. As the condition \rto{bck1} is, due to (\ref{le/to}), essentially an equation, this remains true also for wBCK*-algebras. 
The situation changes when join or meet operation is added.

\begin{theorem}	\label{uppvar}
Let $(A,\vee,1)$ be a semilattice with unit and the natural ordering $\le$, and let $\to$ be a binary operation on $A$. Then the following statements are equivalent:
\begin{thalphalist}
\item $(A,\to,1)$ is a wBCK*-algebra,
\item $\to$ satisfies the conditions {\upshape \rto{bck2}, \rto{Sreg}} and 
\begin{mylist} 
\thitem{b1} $x \to (x \vee y) = 1$,
\thitem{b2} $(x \vee y) \to z \le y \to z$.
\end{mylist}
\end{thalphalist}
\end{theorem} 
\begin{proof}
Evidently, every wBCK*-semilattice satisfies the conditions listed in (b) --- see (\ref{le/to}) and \rto{bck1-}. If, conversely, the conditions are satisfied in $(A, \vee, \to,1)$, then the order relation $\le$ satisfies (\ref{le/to}) in virtue of \rto{bck2}, \rto{Sreg} and (b1), and then \rto{bck1-} follows from (b2) by (\ref{le/to}).
\end{proof}

The next theorem is proved similarly.

\begin{theorem}	\label{lowvar}
Let $(A,\wedge,1)$ be a lower semilattice with unit and the natural ordering $\le$, and let $\to$ be a binary operation on $A$. Then the following statements are equivalent:
\begin{thalphalist}
\item $(A,\to,1)$ is a wBCK*-algebra,
\item $\to$ satisfies the conditions {\upshape \rto{bck2}, \rto{Sreg}} and 
\begin{mylist} 
\thitem{b1} $(x \wedge y) \to y = 1$,
\thitem{b2} $x \to z \le (x \wedge y) \to z$.
\end{mylist}
\end{thalphalist}
\end{theorem}

\begin{corollary} 
The following classes of algebras are equationally definable:
\begin{thalphalist}
\item the class of all upper wBCK*-semilattices,
\item the class of all lower wBCK*-semilattices, as well as its subclasses of weakly contractive and of sectionally j-pseudocomplemented semilattices,
\item the class of all wBCK*-lattices, as well as its subclasses of weakly contractive and of sectionally j-pseudocomplemented lattices,
\item the class of all upper semilattice-ordered pocrigs,
\item the class of all lower semilattice-ordered pocrigs,
\item the class of all lattice-ordered pocrigs.
\end{thalphalist}
\end{corollary}

\begin{proof}
(a) Follows from Theorem \ref{uppvar}.
\par
(b) Follows from Theorem \ref{lowvar}. Note that the condition \rto{reg} can be rewritten in a form of an equation as follows:
\begin{toenum}	
\item	\label{regg}
$x \wedge (x \to y) \le y$.
\end{toenum}
Indeed, \rto{reg} is an easy consequence of \rto{regg}. On the other hand, $x \wedge (x \to y) \le x \to y$, 
and then $x \le x \wedge (x \to y) \to y$ by \rto{wexch}. As $x \wedge (x \to y) \le x$, it follows by \rto{reg} that  $x \wedge (x \to y) \le y$.
\par
(c)
Follows from (a) and (b).
\par
(d),(e) Follow from (a) and (b) respectively, as in semilattices the four conditions listed in Proposition \ref{adj4} are captured by equations.
\par
(f) Follows from (d) and (e).
\end{proof}

For BCK*-semilattices and lattices this was proved by Idziak in  \cite[Theorem 1]{Idz}.
As noted in \cite{Idz}, the variety of upper BCK*-semilattices is neither congruence permutable nor congruence distributive. Clearly, this concerns also upper wBCK*-semi\-lattices. In contrast, the class of lower wBCK*-semilattices is even arithmetical; and so is the class of wBCK*-lattices. Our next theorem together with its proof generalises Theorem 2 of \cite{Idz}.

\begin{theorem}
The variety of lower wBCK*-semilattices is arithmetical.
\end{theorem}
\begin{proof}
The variety is congruence distributive, for it has a majority term
\[
m(x,y,z) := (x \to y. \to y) \wedge (y \to z. \to z) \wedge (z \to x. \to x)
\]
and congruence permutable, for it has a corresponding Mal'cev term
\[
p(x,y,z) := (x \to y. \to z) \wedge (z \to y. \to x)
\]
(see \rto{refl}, \rto{Sreg}, \rto{bck2}, \rto{ubound}). Hence, it is arithmetical.
\end{proof}

Of course, then all subvarieties varieties of lower wBCK*-semilattices and of wBCK*-lattices mentioned in the corollary are also arithmetical.
 
\begin{remark}
Sectionally j-pseudocomplemented lattices are just the j-extensions of sectionally pseudocomplemented lattices mentioned in Introduction. Another equational description of such extensions was presented in \cite[Theorem 2]{Ch}. It was stated in Theorems 5.1 and 5.3 of \cite{ChR1} that this variety is arithmetical and 1-regular. The easy proof of the latter 
theorem goes even for any variety of upper wBCK*-semilattices.
\end{remark}

The next theorem and its proof are suggested by the similar result \cite[Theorem 3]{Idz} for pocrims.

\begin{theorem}
The variety of uppersemilattice-ordered pocrigs is arithmetical.
\end{theorem}
\begin{proof} 
The corresponding Mal'cev terms are 
\[
m(x,y,z) := x(x \to y) \vee y(y\to z) \vee z(z\to x),
\]
and
\[p(x,y,z) := x(y \to z) \vee z(y \to x).
\]
\end{proof}

\end{document}